\documentclass[11pt]{article}
\usepackage{amsfonts,amssymb,amsmath,amsthm}
\usepackage{enumitem}
\usepackage{color} % load the color package
   \definecolor{cites}{rgb}{0.50 , 0.00 , 0.00}  % colour for citations
   \definecolor{urls} {rgb}{0.00 , 0.00 , 0.50}  % colour for URL's
   \definecolor{links}{rgb}{0.00 , 0.00 , 0.50}   % colour for links
\usepackage[
      colorlinks=true,   % Yes, we want coloured links!
      citecolor=cites,   % Set the citations colour,
      urlcolor=urls,     % and the URL's colour,
      linkcolor=links,   % and the colour for the links.
      pdftitle={A note on Hausdorff convergence of pseudospectra},  % Put your title here
      pdfauthor={Marko Lindner, Dennis Schmeckpeper}, % And your name here
      pdfpagemode=UseOutlines, % None, UseThumbs, UseOutlines, FullScreen
      pdfstartview=FitH,       % Fit, FitH, FitB
      bookmarksopen=false      % Bookmark tree already opened?
   ]{hyperref}

\newcommand\eps\varepsilon
\newcommand\ph\varphi
\newcommand\spec{{\rm spec}\,}  % with space
\newcommand\specn{{\rm spec}}   % no space
\newcommand\speps{{\rm spec}_\eps}
\newcommand\Speps{{\rm Spec}_\eps}
\newcommand\dH{d_{\rm H}}

\newcommand\dist{{\rm dist}}

\newcommand\closn{{\rm clos}}
\newcommand\clos{\closn\,}
\newcommand\sub{{\rm sub}}

\newcommand\Hto{
   \unitlength0.1ex
   \begin{picture}(30,15)
   \put(13,16){\makebox(0,0)[]{\tiny\rm H}}
   \put(15,5){\makebox(0,0)[]{$\to$}}
   \end{picture}
}
\newcommand\Hot{
   \unitlength0.1ex
   \begin{picture}(30,15)
   \put(17,16){\makebox(0,0)[]{\tiny\rm H}}
   \put(15,5){\makebox(0,0)[]{$\leftarrow$}}
   \end{picture}
}

\newcommand\C{{\mathbb C}}
\newcommand\R{{\mathbb R}}

\newcommand\Z{{\mathbb Z}}
\newcommand\N{{\mathbb N}}

\newcommand\D{{\mathbb D}}

\newtheorem{theorem}{Theorem}[section]
\newtheorem{lemma}[theorem]{Lemma}
\newtheorem{corollary}[theorem]{Corollary}

\newenvironment{example}
 {\par\noindent\refstepcounter{theorem}{\bf Example \thetheorem}\ }
 {\raisebox{1mm}{\framebox{}}\pagebreak[2]}

\newtheorem{remarkx}[theorem]{Remark}
\newenvironment{remark}
  {\pushQED{\qed}\remarkx\normalfont}
  {\popQED\endremarkx}

\numberwithin{figure}{section}  % needs package amsmath

\newcounter{abccounter}

\makeatletter
\let\@fnsymbol\@arabic
\makeatother
\begin{document}
\title{\bf A note on Hausdorff convergence of pseudospectra}
\author{
{\sc Marko Lindner}\footnote{Maths Institute, TU Hamburg, Germany, \href{mailto:lindner@tuhh.de}{\tt lindner@tuhh.de}}
\quad and\quad
{\sc Dennis Schmeckpeper}\footnote{Maths Institute, TU Hamburg, Germany,
\href{mailto:dennis.schmeckpeper@tuhh.de}{\tt dennis.schmeckpeper@tuhh.de}}

}

% \date{\today}
\maketitle
\begin{quote}
\renewcommand{\baselinestretch}{1.0}
\footnotesize {\sc Abstract.}
For a bounded linear operator on a Banach space, we study approximation of the spectrum and pseudospectra in the Hausdorff distance. We give sufficient and necessary conditions in terms of pointwise convergence of appropriate spectral quantities.
\end{quote}

\section{Introduction}
{\bf Spectrum, pseudospectrum and lower norm.}
Given a bounded linear operator $A$ on a Banach space $X$, we denote its {\sl spectrum} and {\sl pseudospectra} \cite{TrefEmb}, respectively, by
\[
\spec A\ :=\ \{\lambda\in\C: A-\lambda I \text{ is not invertible}\}
\]
and
\begin{equation} \label{eq:speps}
\textstyle
\speps A\ :=\ \{\lambda\in\C: \|(A-\lambda I)^{-1}\|>\frac 1\eps \},\qquad \eps>0,
\end{equation}
where we identify $\|B^{-1}\|:=\infty>\frac 1\eps$ if $B$ is not invertible, so that $\spec A\subseteq \speps A$ for all $\eps>0$.

A fairly convenient access to the norm of the inverse is given by the so-called {\sl lower norm}, the number
\begin{equation}\label{eq:nu}
\nu(A)\ :=\ \inf_{\|x\|=1}\|Ax\|.
\end{equation}
Indeed, putting $\mu(A):=\min\{\nu(A),\,\nu(A^*)\}$, we have
\begin{equation}\label{eq:invmu}
\|A^{-1}\|\ \ =\ 1/\mu(A),
\end{equation}
where $A^*$ is the adjoint on the dual space $X^*$ and equation \eqref{eq:invmu} takes  the form $\infty=1/0$ if and only if $A$ is not invertible. 

One big advantage of this approach is that, in case $X=\ell^p(\Z^d,Y)$ with $p\in [1,\infty]$, $d\in\N$ and a Banach space $Y$, $\nu(A)$ can be approximated by the same infimum \eqref{eq:nu} with $x\in X$ restricted to elements with finite support of given diameter $D$. We can even quantify the approximation error against $D$, see \cite{{CW.Heng.ML:UpperBounds}} and \cite{LiSei:BigQuest} (as well as \cite{HagLiSei} for a corresponding result on the norm).

By means of \eqref{eq:invmu}, we can rewrite spectrum and pseudospectra as follows:
\[
\spec A\ =\ \{\lambda\in\C: \mu(A-\lambda I)=0\}
\]
and
\[
\speps A\ =\ \{\lambda\in\C: \mu(A-\lambda I)<\eps \},\qquad \eps>0.
\]

In other words, $\spec A$ is the level set  of the function $f:\C\to [0,\infty)$ with
\begin{equation}\label{eq:f}
f(\lambda)\ :=\ \mu(A-\lambda I)
\end{equation}
for the level zero, and $\speps A$ is the sublevel set of $f$ for the level $\eps>0$.

{\bf Sublevel sets.}
For a function $g:\C\to [0,\infty)$ and $\eps>0$, let
\[
\sub_\eps(g) := \{\lambda\in\C : g(\lambda)<\eps\}
\]
denote the {\sl sublevel set of $g$ for the level $\eps$}.

In general, pointwise convergence $g_n\to g$ of functions $\C\to [0,\infty)$ need not coincide with  Hausdorff convergence of their sublevel sets:

\begin{example} \label{ex:1}
Suppose we have $g$ and $g_n$ such that a) $g_n\to g$ as well as b) $\sub_\eps(g_n)\ \Hto\ \sub_\eps(g)$ hold for all $\eps>0$. Increasing $g(\lambda)$ to a certain level $\eps>0$ in a point $\lambda$, where $g$ was continuous and below $\eps$ before, changes the state of a), while it does not affect $\closn(\sub_\eps(g))$ and hence b). 
\end{example}

So let us look at continuous examples from here on.

\begin{example} \label{ex:2}
For $g_n(\lambda):=\frac{|\lambda|}n\to 0=:g(\lambda)$, the Hausdorff distance of
the sublevel sets 
\[
\sub_\eps(g_n)=n\eps\D
\qquad\text{and}\qquad
\sub_\eps(g)=\mathbb C
\]
remains infinite, where $\D$ denotes the open unit disk in $\C$.  
\end{example}

Of course, this problem was due to the unboundedness of $\sub_\eps(g)$.
So let us further focus on functions that go to infinity at infinity, so that all
sublevel sets are bounded.
\begin{example}{\bf (locally constant)\ } \label{ex:3}
Let $g(\lambda):=h(|\lambda|)$
and $g_n(\lambda):=h_n(|\lambda|)$ for $n\in\N$, where
\[
\textstyle
h(x):=\max\{\min\{|x|,1\},|x|-1\},\qquad
h_1(x):= \frac 14 x^2,\qquad
h_n:=h+\frac 1n(h_1-h)\to h
\]
for $x\in\R$ and $n\in\N$. 
Then, unlike any $g_n$, $g$ is locally constant in $2\D\setminus\D$. Consequently,
\[
\sub_1(g_n)\equiv 2\D\ \not\!\!\!\Hto\ \D=\sub_1(g)
\qquad\text{but}\qquad g_n\to g.
\vspace{-18pt}
\]
~\hfill \qedhere
\end{example}
%
%\begin{example} \label{ex:4}
%Wir betrachten auf $[0,1]$ die Funktionenfolge $f_n (x) = \left|x \sin\left(\frac{\pi}{x^n}\right)\right|$. Diese konvergiert offensichtlich nicht punktweise. Insbesondere also auch nicht gegen $f = 0$. Es gilt aber für alle $\varepsilon > 0$, dass 
%\[
%\sub_\eps(f_n) \to [0,1] = sub_\eps(f). 
%\]
%{\tt Beweis...}
%\end{example}

\begin{example} {\bf (increasingly oscillating)\ } \label{ex:5}
Let $g_n(\lambda):=h_n(|\lambda|)$ for $n\in\N$, where
\[
h_n(x):=\left\{
\begin{array}{ll}
|\sin(n\pi x)|,&x\in[0,1],\\
x-1,&x>1.
\end{array}
\right.
\] 
Then $h_n(x)<\eps$ for all $\eps>0$ and
\[
\textstyle
x\in \frac 1n\Z\cap [0,1]\Hto[0,1]
\quad\text{as}\quad n\to\infty.
\]
It follows that $\sub_\eps(g_n)\Hto(1+\eps)\D$ for all $\eps>0$, while $g_n$ does not converge pointwise at all.
\end{example}

{\bf The result.}
For a sequence of bounded operators $A_n$ on $X$ and their corresponding functions $f_n:\C\to [0,\infty)$ with
\begin{equation}\label{eq:fn}
f_n(\lambda)\ :=\ \mu(A_n-\lambda I),\qquad n\in\N,
\end{equation}
we show equivalence of pointwise convergence $f_n\to f$ and Hausdorff convergence of their sublevel sets, i.e.~of the corresponding pseudospectra,
\[
f_n\to f
\qquad\iff\qquad
\speps A_n\ \Hto\ \speps A,\quad\forall\eps>0.
\]
This result is not surprising (and similar arguments have been used e.g.~in \cite{Colbrook:PE} in a more specific situation) but there are some little details that deserve to be written down as this separate note.

In \cite{subwords}, the approximation of the lower norm of $H(b) -\lambda I$ for
a (generalized) discrete Schrödinger operator $H(b)$ and $\lambda \in
\mathbb{C}$, is established via successive exhaustion of the set of finite
subwords of the potential $b \in \ell^\infty(\mathbb{Z})$.
Together with our paper here, this yields Hausdorff approximation of the pseudospectrum of $H(b)$, also see \cite{stabindic}.

\section{Lipschitz continuity and non-constancy of $\mu$}
Our functions $\nu$ and $\mu$, and hence $f$ and $f_n$, have two properties that rule out effects as in Examples \ref{ex:1} -- \ref{ex:5}: Lipschitz continuity and the fact that their level sets have no interior points, i.e.~$\mu$ is not constant on any open sets.

The first property is straightforward but the latter is a very nontrivial subject \cite{Globevnik,Shargorodsky08,Shargorodsky09,ShargoSkarin,DaviesShargo}, and it actually limits the choice of our Banach space $X$ as shown in Lemma \ref{lem:nonconst}.

\begin{lemma} \label{lem:Lipschitz}
For all bounded operators $B,C$ on $X$, one has
\[
|\nu(B)-\nu(C)|\ \le\ \|B-C\|,
\]
so that also $\mu(B)=\min\{\nu(B),\nu(B^*)\}$ is Lipschitz continuous with Lipschitz constant $1$.\\
The same follows for the functions $f$ from \eqref{eq:f} and $f_n$ with $n\in\N$ from \eqref{eq:fn}.
\end{lemma}
This result is absolutely standard but we give the (short) proof, for the reader's convenience:
\begin{proof}
For all $x\in X$ with $\|x\|=1$, one has
\[
\|B-C\|\ \ge\ \|Bx-Cx\|\ \ge\ \|Bx\|-\|Cx\|\ \ge\ \nu(B)-\|Cx\|.
\]
Now pass to the infimum in $\|Cx\|$ to get $\|B-C\|\ge\nu(B)-\nu(C)$. Finally swap $B$ and $C$.
\end{proof}

It will become crucial to understand when the resolvent norm of a bounded
operator cannot be constant on an open subset. This is a surprisingly rich and
deep problem. As it turns out it is connected to a geometric property,
the complex uniform convexity, of the underlying Banach space (see \cite[Definition
2.4 (ii)]{Shargorodsky08}).
\pagebreak % sorry, dirty pimping of page layout

\begin{lemma}[{Globevnik \cite{Globevnik}, Shargorodsky et al. \cite{Shargorodsky08,Shargorodsky09,ShargoSkarin,DaviesShargo}}] \label{lem:nonconst}
~\\Let $X$ be a Banach space which satisfies at least one of the following properties,\\[-8mm]
\begin{enumerate}[label=(\alph*)] \itemsep-1mm
\item $\dim(X) < \infty$,
\item $X$ is complex uniform convex, %(see \cite{Globevnik} or \cite{Shargorodsky08}),
\item its dual, $X^*$, is complex uniform convex.
\end{enumerate}
~\\[-8mm]
For example, every Hilbert space is of this kind, and every space $X=\ell^p(\Z^d,Y)$ with $p\in [1,\infty]$ and $d\in\N$ falls into this category as soon as $Y$ does \cite{Day}. Then, for every bounded operator $A$ on $X$, the resolvent norm, 
\[
\lambda\ \mapsto\  \|(A-\lambda I)^{-1}\|=1/\mu(A - \lambda I),
\]
cannot be locally constant on any open set in $\C$, and, consequently,
\[
\forall \varepsilon > 0: \quad \closn(\speps A) = \{ \lambda \in
\mathbb{C} : \mu (A -\lambda I) \leq \varepsilon \}.
\]
\end{lemma}

\section{Set sequences and Hausdorff convergence}
Let $(S_n)$ be a sequence of bounded sets in $\mathbb C$
and recall the following notations (e.g.~\cite[\S 3.1.2]{HaRoSi2}):
\begin{itemize} \itemsep-1mm
\item $\liminf S_n=$ the set of all limits of sequences $(s_n)$ with $s_n\in S_n$;
\item $\limsup S_n=$ the set of all partial limits of sequences $(s_n)$ with $s_n\in S_n$;
\item both sets are always closed;
\item let us write $S_n\to S$ if $\liminf S_n=\limsup S_n=S\,(=\clos\, S)$;
\item then $S_n\to S\iff \clos\,S_n\to S$ (and again, $S=\clos\, S$ is automatic).
\end{itemize}

Here is an apparently different approach to set convergence:
For $z\in\C$ and $S\subseteq\C$, set $\dist(z,S):=\inf_{s\in S}|z-s|$.
The {\sl Hausdorff distance} of two bounded sets $S,T\subseteq\C$, defined via
\[
\dH(S,T)\ :=\ \max\left\{\sup_{s\in S}\dist(s,T),\ \sup_{t\in T}\dist(t,S)\right\},
\]
\begin{itemize} \itemsep-1mm
\item ...is a metric on the set of all compact subsets of $\C$;
\item ...is just a pseudometric on the set of all bounded subsets of $\C$:\\
besides symmetry and triangle inequality, one has $\dH(S,T)=0\iff \clos S=\clos T$ since
\[
\dH(S,T)\ =\ \dH(\clos S, T)\ =\ \dH(S,\clos T)\ =\ \dH(\clos S,\clos T);
\]
\item let us still write $S_n\Hto S$ if $\dH(S_n,S)\to 0$, also for merely bounded sets $S_n,S$;
\item the price is that the limit $S$ in $S_n\Hto S$ is not unique:\\
one has $S_n\Hto S$ and $S_n\Hto T$ if and only if $\dH(S,T)=0$, 
i.e.~$\clos S=\clos T$.
\end{itemize}

Both notions of set convergence are connected, via the Hausdorff theorem:
\begin{equation}\label{eq:Haus}
S_n\Hto S\qquad\iff\qquad S_n\to \clos S\,.
\end{equation}
Remember:\\[-8mm]
\begin{itemize} \itemsep-1mm
\item the limit of ``$\to$'' is always closed;
\item the limit of ``$\Hto$'' need not be closed...
\item ... but its uniqueness only comes by passing to the closure;
\item passing to the closure in front of ``$\to$'' or ``$\Hto$'' does not change the statement.
\end{itemize}
\pagebreak % sorry, dirty pimping of page layout

\begin{lemma} \label{lem:X}
Let $S_n$ and $T_n$ be bounded subsets of $\C$ with
$S_n \to S$ and $T_n \to T$. \\
In addition, suppose $S_n \setminus T_n\ne\varnothing$. Then:
\begin{enumerate}[label=(\alph*)]\itemsep-1mm
\item In general, it does \underline{not} follow that
\[
S_n\setminus T_n \quad \to\quad S\setminus T.
\]
\item However, it always holds that
\[
\liminf (S_n\setminus T_n) \quad\supseteq\quad S\setminus T.
\]
\end{enumerate}
\end{lemma}
\begin{proof}
\begin{enumerate}[label=(\alph*)]\itemsep-1mm
\item Consider
$S_n := [0,1]\ \to\ [0,1] =:S$ and $T_n:=\frac{1}{n}\mathbb{Z}\cap [0,1]\ \to\ [0,1] =:T$.\\ 
Then $S_n\setminus T_n \ \to\ [0,1] \ne \varnothing = S\setminus T$.
\item Let $x \in S\setminus T$.
Since $x\in S$, there is a sequence $(x_n)$ with $x_n\in S_n$ such that $x_n \to x$.\\
We show that $x_n\not\in T_n$, eventually.\\
Suppose $x_n \in T_n$ for infinitely many $n\in\N$. 
Then there is a strictly monotonic sequence $(n_k)$ in $\N$ with $x_{n_k}\in T_{n_k}$. 
But then
\[
x=\lim_n x_n =\lim_k x_{n_k}\in \limsup_n T_n = T,
\]
which contradicts $x\in S\setminus T$.

Consequently, just finitely many elements of the sequence $(x_n)$ can be in $T_n$. 
Replacing these by elements from $S_n\setminus T_n$ does not change the limit, $x$.
So $x\in \liminf (S_n\setminus T_n)$. \qedhere
\end{enumerate}
\end{proof}

\section{Equivalence of pointwise convergence $f_n\to f$ and Hausdorff convergence of the pseudospectra}
Here is our main theorem. Note that we do not require any convergence of $A_n$ to $A$.
\begin{theorem} \label{thm:main}
Let $X$ be a Banach space with the properties from Lemma \ref{lem:nonconst}
and let $A$ and $A_n,\ n\in\N$, be bounded linear operators on $X$. 
Then the following are equivalent for the functions and sets introduced in 
\eqref{eq:speps}, \eqref{eq:f} and \eqref{eq:fn}:%\\[-3mm]
\[
\begin{array}{rlp{50mm}}
(i) & f_n\to f\text{ pointwise},&\\
(ii) & \text{for all }\eps>0,\text{ one has }\speps A_n\ \Hto\ \speps A.&
\end{array}
\]
%\begin{enumerate}[label=(\roman*)]\itemsep-1mm
%\item $f_n\to f$ pointwise,
%\item for all $\eps>0$, one has $\speps A_n\ \Hto\ \speps A$.
%\end{enumerate}
\end{theorem}
\begin{proof}
$(i)\implies (ii)$: 
Assume $(i)$ and take $\eps>0$.
For $f(\lambda)<\eps$, $(i)$ implies $f_n(\lambda)<\eps\ \forall n\ge n_0$. 
So it follows
\[
\speps A\ \subseteq\ \liminf\speps A_n\ \subseteq\ \limsup\speps A_n.
\]
Now let $\lambda\in\limsup\speps A_n$, i.e.~$\lambda=\lim \lambda_{n_k}$ 
with $\lambda_{n_k}\in\speps A_{n_k}$, so that $f_{n_k}(\lambda_{n_k})<\eps$. 
Then
\[
|f(\lambda)-f_{n_k}(\lambda_{n_k})|\ \le\ 
\underbrace{|f(\lambda)-f_{n_k}(\lambda)|}_{\to 0\text{ by }(i)}
\ +\ \underbrace{|f_{n_k}(\lambda)-f_{n_k}(\lambda_{n_k})|}_{\le|\lambda-\lambda_{n_k}|\to 0}\ \to\ 0.
\]
Consequently, $f(\lambda)\le\eps$ and hence $\lambda\in \closn(\speps A)$, by Lemma \ref{lem:nonconst}. We get
\[
\speps A\ \subseteq\ \liminf\speps A_n\ \subseteq\ \limsup\speps A_n\ \subseteq\ \closn(\speps A).
\]
Passing to the closure everywhere in this chain of inclusions, just changes $\speps A$ at the very left into $\closn(\speps A)$, and we have $\speps A_n\to\closn(\speps A)$ and hence, by \eqref{eq:Haus}, $(ii)$.
\pagebreak % sorry, dirty pimping of page layout

$(ii)\implies (i):$  Take $\lambda\in\C$ and put $\eps:=f(\lambda)$.\\[-7mm]
\begin{itemize}
\item Case 1: $\eps=0$.\\
Take an arbitrary $\delta>0$. 
Then, by $(ii)$, $\lambda\in \specn_\delta A\Hot \specn_\delta A_n$. 
So, by \eqref{eq:Haus}, there is a sequence $(\lambda_n)_{n\in\N}$ 
with $\lambda_n\in \specn_\delta A_n=f_n^{-1}([0, \delta))$ and $\lambda_n \to \lambda$. 
    
\item Case 2: $\eps>0$.\\
Take an arbitrary $\delta\in (0,\eps)$. By $(ii)$, we have
\[
S_n\ :=\ \specn_{\eps+\delta}A_n\ \Hto\  \specn_{\eps+\delta}A\ =:\ S,
\qquad i.e.~S_n\to\clos S, \text{ by } \eqref{eq:Haus},
\]
and
\[
T_n\ :=\ \specn_{\eps-\delta}A_n\ \Hto\  \specn_{\eps-\delta}A\ =:\ T,
\qquad i.e.~T_n\to\clos T, \text{ by } \eqref{eq:Haus}.
\]
By Lemma \ref{lem:X} (b), 
\[
\lambda\quad\in\quad 
\closn(\specn_{\eps+\delta} A)\ \setminus\ \closn(\specn_{\eps-\delta}A)
\quad \subseteq\quad 
\liminf \Big(\specn_{\eps+\delta} A_n\ \setminus\ \specn_{\eps-\delta}A_n\Big),
\]  
in short:
\[
\lambda\quad\in\quad
f^{-1}\big((\eps-\delta,\eps+\delta]\big)\quad \subseteq\quad 
\liminf f_n^{-1}\big([\eps-\delta,\eps+\delta)\big).
\]
So there is a sequence $(\lambda_n)_{n\in\N}$ with 
$\lambda_n\in f_n^{-1}\big([\eps - \delta,\eps +\delta)\big)$ and $\lambda_n \to \lambda$. 
 \end{itemize} 
~\\[-6mm]  % sorry, dirty pimping of page layout
In both cases, we conclude
\[
|f(\lambda)-f_n(\lambda)|
\ \le\ \underbrace{|\overbrace{f(\lambda)}^\eps - f_n(\lambda_n)|}_{\le\delta}
\ +\ \underbrace{|f_n(\lambda_n)-f_n(\lambda)|}_{\le |\lambda_n-\lambda|\to 0}\ <\ 2\delta
\]
for all sufficiently large $n$, and hence $f_n(\lambda)\to f(\lambda)$ as $n\to\infty$, i.e.~$(i)$.
\end{proof}

\begin{corollary}
Let the assumptions of Theorem \ref{thm:main} be satisfied.\\
If $X$ is a Hilbert space and the operators $A$ and $A_n$, $n\in\N$, are normal then $(i)$ and $(ii)$ are also equivalent to
\[
 \spec A_n\ \Hto\ \spec A.
\]
\end{corollary}
\begin{proof}
For normal operators, the $\eps$-pseudospectrum is exactly the $\eps$-neighborhood of the spectrum, e.g.~\cite{TrefEmb}. But $B_\eps(S_n)\Hto B_\eps(S)$ for all $\eps>0$ implies $S_n\Hto S$.
\end{proof}

\begin{remark}
{\bf a) }  The pointwise convergence $f_n\to f$ is uniform on compact subsets of $\C$. (Take an $\frac \eps 3$-net for the compact set and use the uniform Lipschitz continuity of the $f_n$.) 

{\bf b) }
It is well-known \cite{TrefEmb} that $\speps A\subseteq r\D$ with $r=\|A\|+\eps$. So if $(A_n)_{n\in\N}$ is a bounded sequence then $\speps B\subset r\D$ for all $B\in\{A,A_n:n\in\N\}$ with $r=\max\{\|A\|,\sup \|A_n\|\}+\eps$. By a), the convergence $f_n\to f$ is uniform on $\closn(r\D)$.
\end{remark}

\begin{remark}
Sometimes (especially in earlier works), pseudospectra are defined in terms of non-strict inequality:
\[
\textstyle
\Speps A\ :=\ \{\lambda\in\C: \|(A-\lambda I)^{-1}\|\ge\frac 1\eps \},\qquad \eps>0.
\]
One benefit is to get compact pseudospectra, in which case $\dH$ is a metric and $\Hto$ has a unique limit.
By Lemma \ref{lem:nonconst}, $\Speps A=\closn(\speps A)$ for all $\eps>0$. But since $S_n\Hto S $ if and only if $\closn(S_n)\Hto\closn(S)$, one could add, if one prefers, this third equivalent statement to Theorem~\ref{thm:main}:
\[
(iii)\quad\Speps A_n\ \Hto\ \Speps A,\qquad \forall\eps>0. \qedhere
\]
\end{remark}

%%%%%%%%%%%%%%%%%%%%%%%%%%%%%%%%%%%%%%%%%%%%%%%%
\medskip

{\bf Acknowledgements.}
The authors thank Fabian Gabel and Riko Ukena from TU Hamburg for helpful comments and discussions.
%\pagebreak

\vfill
\noindent {\bf Authors' addresses:}\\
\\
Marko Lindner\hfill \href{mailto:lindner@tuhh.de}{\tt lindner@tuhh.de}\\
Dennis Schmeckpeper\hfill \href{mailto:dennis.schmeckpeper@tuhh.de}{\tt dennis.schmeckpeper@tuhh.de}\\[2mm]
Institut Mathematik\\
TU Hamburg (TUHH)\\
D--21073 Hamburg\\
GERMANY
\end{document}